\documentclass{amsart}
\usepackage{amssymb}
\usepackage[utf8]{inputenc}
\usepackage{hyperref}
\usepackage{bm}
\usepackage{comment}

\newcommand{\Z}{{\ensuremath{\mathbb{Z}}}}

\theoremstyle{definition}
\newtheorem{thm}{Theorem}[section]

\newtheorem{cor}[thm]{Corollary}
\newtheorem{defi}[thm]{Definition}

\newtheorem{lem}[thm]{Lemma}
\newtheorem{rem}[thm]{Remark}
\newtheorem{prp}[thm]{Proposition}

\title[Images of graded polynomials on upper triangular matrices]{Images of multilinear graded polynomials on upper triangular matrix algebras}
\author[P. Fagundes]{Pedro Fagundes}
\address{IMECC, Universidade Estadual de
		Campinas, Rua S\'ergio Buarque de Holanda, 651, Cidade Universit\'aria ``Zeferino Vaz'', Distr. Bar\~ao Geraldo, Campinas, S\~ao Paulo, Brazil, CEP
		13083-859}\email{pedro.fagundes@ime.unicamp.br}
\author[P. Koshlukov]{Plamen Koshlukov}
\address{IMECC, Universidade Estadual de
		Campinas, Rua S\'ergio Buarque de Holanda, 651, Cidade Universit\'aria ``Zeferino Vaz'', Distr. Bar\~ao Geraldo, Campinas, S\~ao Paulo, Brazil, CEP
		13083-859}\email{plamen@unicamp.br}	

\begin{document}

\maketitle

\begin{abstract}
    In this paper we study the images of multilinear graded polynomials on the graded algebra of upper triangular matrices $UT_n$. For positive integers $q\leq n$, we classify these images on $UT_{n}$ endowed with a particular elementary $\Z_{q}$-grading. As a consequence, we obtain the images of multilinear graded  polynomials on $UT_{n}$ with the natural $\Z_{n}$-grading. We apply this classification in order to give a new condition for a multilinear polynomial in terms of graded identities so that to obtain the traceless matrices in its image on the full matrix algebra. We also describe the images of multilinear polynomials on the graded algebras $UT_{2}$ and $UT_{3}$, for arbitrary gradings. We finish the paper by proving a similar result for the graded Jordan algebra $UJ_{2}$, and also for $UJ_{3}$ endowed with the natural elementary $\Z_{3}$-grading. 
    
\end{abstract}

\section{Introduction}
Let $\mathcal{A}$ be an associative algebra over a field $F$, and let $f\in F\langle X\rangle$ be a multilinear polynomial from the free associative algebra $F\langle X\rangle$. Lvov posed the question to determine whether the image of a multilinear $f$ when evaluated on $\mathcal{A}=M_n(F)$, is always a vector subspace of $M_n(F)$, see \cite[Problem 1.93]{dnestr}. The original question is attributed to Kaplansky and asks the determination of the image of a polynomial $f$ on $\mathcal{A}$. It is well known that the above question is equivalent to that of determining whether the image of a multilinear $f$ on $M_n(F)$ is 0, the scalar matrices, $sl_n(F)$ or $M_n(F)$. Clearly the first possibility corresponds to $f$ being a polynomial identity on $M_n(F)$, and the second gives the central polynomials. 

Recall here that the description of all polynomial identities on $M_n(F)$ is known only for $n\le 2$, see \cite{razmal, dral} for the case when $F$ is of characteristic 0, and \cite{pkm2} for the case of $F$ infinite of characteristic $p>2$. The same holds for the central polynomials  \cite{okhitin, jcpk}. The theorem of Amitsur and Levitzki gives the least degree polynomial identity for $M_n(F)$, the standard polynomial $s_{2n}$, see \cite{am_lev}. Recall also that one of the major breakthroughs in PI theory was achieved by Formanek and by Razmyslov \cite{for, razm} who proved the existence of nontrivial (that is not identities) central polynomials for the matrix algebras. As for $sl_n(F)$, a theorem of Shoda \cite{shoda} gives that every $n\times n$ matrix over a field of zero characteristic is the commutator of two matrices; later on Albert and Muckenhoupt \cite{albert} generalized this to arbitrary fields. Hence all four conjectured possibilities for the image of a multilinear polynomial on $M_n(F)$ can be achieved. 

The study of images of polynomials on the full matrix algebra is of considerable interest for obvious reasons, among these the relation to polynomial identities. In \cite{belov1} the authors settled the conjecture due to Lvov in the case of $2\times 2$ matrices over a quadratically closed field $F$ (that is if $f$ is a given polynomial in several variables then every polynomial in one variable of degree $\le 2\deg f$ over $F$ has a root in $F$). They proved that for every multilinear polynomial $f$ and for every field $F$ that is quadratically closed with respect to $f$, the image of $f$ on $M_2(F)$ is 0, $F$, $sl_2(F)$ or $M_2(F)$. It should be noted that the authors in \cite{belov1} proved a stronger result. Namely they considered a so-called semi-homogeneous polynomial $f$: a polynomial in $m$ variables $x_1$, \dots, $x_m$ of weights $d_1$, \dots, $d_m$ respectively such that every monomial of $f$ is of (weighted) degree $d$ for a fixed $d$. They proved that the image of such a polynomial on $M_2(F)$ must be 0, $F$, $sl_2(F)$, the set of all non-nilpotent traceless matrices, or a dense subset of $M_2(F)$. Here the density is according to the Zariski topology. Later on in \cite{malev} the author gave the solution to the problem for $2\times 2$ matrices for the case when $F$ is the field of the real numbers. In the case of $3\times 3$ matrices the known results can be found in \cite{belov2}. The images of polynomials on $n\times n$ matrices for $n>3$ are hard to describe, and there are only partial results, see for example \cite{belov3}. Hence in the case of $n\times n$ matrices one is led to study images of polynomials of low degree. Interesting results in  this direction are due to \v Spenko \cite{spenko}, she proved the conjecture raised by Lvov in case $F$ is an algebraically closed field of characteristic 0, and $f$ is a multilinear Lie polynomial of degree at most 4. Further advances in the field were made in  \cite{br_kl1, br_kl2, br}. In \cite{br} the author proved that if $A$ is an algebra over an infinite field $F$ and $A=[A,A]$ then the image of an arbitrary polynomial which is neither an identity nor a central polynomial, equals $A$. Recently Malev \cite{malev_quat} described completely the images of multilinear polynomials on the real quaternion algebra; he also described the images of semi-homogeneous polynomials on the same algebra.

If the base field $F$ is finite, a theorem of Chuang \cite{chuang} states that the image of a polynomial without constant term can be every subset of $M_n(F)$ that contains $0$ and is closed under conjugation by invertible matrices. In the same paper it was also shown that such a statement fails when $F$ is infinite. 

Therefore it seems likely it should be very difficult to describe satisfactory the images of multilinear polynomials on $M_n(F)$. That is why people started studying images of polynomials on ``easier" algebras, and also on algebras with an additional structure. The upper triangular matrix algebras are quite important in Linear Algebra because of their applications to different branches of Mathematics and Physics. They are also very important in PI theory: they describe, in a sense, the subvarieties of the variety of algebras generated by $M_2(F)$ in characteristic 0. Block-triangular matrices appear in the description of the so-called minimal varieties of algebras. The images of polynomials on the upper triangular matrices have been studied rather extensively. In \cite{wang} the author described the images of multilinear polynomials on $UT_2(F)$, the $2\times 2$ upper triangular matrices over a field $F$. The images of multilinear polynomials of degree up to 4 on $UT_n=UT_n(F)$ for every $n$ were classified in \cite{mello_fag}, and in \cite{Fag} the first named author of the present paper described the images of arbitrary multilinear polynomials on the strictly upper triangular matrices. It turned out that if $f$ is a multilinear polynomial of degree $m$ then its image on the strictly upper triangular matrix algebra $J$ is either 0 or $J^{m}$. The following conjecture was raised in \cite{mello_fag}: Is the image of a multilinear polynomial on $UT_n(F)$ always a vector subspace of $UT_n(F)$? This conjecture was solved independently in \cite{LWa}, for infinite fields (or finite fields with sufficiently many elements), and in \cite{GMe}. Further results concerning images of polynomials on the upper triangular matrix algebra can be found in \cite{tcm, zw, wzl}.

Gradings on algebras appeared long ago; the polynomial ring in one or several variables is naturally graded by the infinite cyclic group $\mathbb{Z}$ by the degree. Gradings on algebras by finite groups are important in Linear Algebra and also in Theoretical Physics: the Grassmann (or exterior) algebra is naturally graded by the cyclic group of order 2, $\mathbb{Z}_2$. In fact the Grassmann algebra is the most well-known example of a superalgebra. It should be noted that while in the associative case, the term ``superalgebra" is synonymous to ``$\textbf{Z}_2$-graded algebra", if one considers nonassociative algebras these notions are very different: a Lie or a Jordan superalgebra seldom is a Lie or a Jordan algebra. We are not going to discuss further such topics because these are not relevant for our exposition.

In \cite{wall}, Wall classified the finite dimensional $\mathbb{Z}_2$-graded algebras that are graded simple. Later on the description of all gradings on matrix algebras was obtained as well as on simple Lie and Jordan algebras. We refer the readers to the monograph \cite{ek} for the state-of-art and for further references. In PI theory gradings appeared in the works of Kemer, see \cite{kemer}, and constituted one of the main tools in the classification of the ideals of identities of associative algebras, which in turn led him to the positive solution of the long-standing Specht Problem. It turns out that the graded identities are easier to describe than the ordinary ones; still they provide a lot of information on the latter. It is somewhat surprising that the images of polynomials have not been studied extensively in the graded setting. 
In \cite{Kul}, Kulyamin described the images of graded polynomials on matrix algebras over the group algebra of a finite group over a finite field. 

The upper triangular matrix algebra admits various gradings, these were shown to be isomorphic to elementary ones, see \cite{VZa}. A grading on a subalgebra $A$ of $M_n(F)$ is elementary if all matrix units $e_{ij}\in A$ are homogeneous in the grading. All elementary gradings on $UT_n$ were classified in \cite{VKV}; in the same paper the authors described the graded identities for all these gradings. In this paper we fix an arbitrary field $F$ and the upper triangular matrix algebra $UT_n$. 

In Section \ref{sect3} we prove that for an arbitrary group grading on $UT_n$, $n>1$, there are no nontrivial graded central polynomials. (Hence the image of a graded polynomial on $UT_n$ cannot be equal to the scalar matrices whenever $n>1$.) In Section \ref{sect4} we consider a specific grading on $UT_n$, and describe all possible images of multilinear graded polynomials for that grading. It turns out that the images are always homogeneous vector subspaces. We impose a mild restriction on the cardinality of the base field. As a by-product of the proof we obtain a precise description of the graded identities for this grading.

In Section \ref{sect5} we give a sufficient condition for the traceless matrices to be contained in the image of a multilinear graded polynomial. Once again we require a mild condition on the cardinality of the field. Section \ref{sect6} studies the graded algebras $UT_2$ and $UT_3$. We prove that the image of a multilinear graded  polynomial on $UT_2$, for every group grading, is a homogeneous subspace. In the case of $UT_3$, the image of such a polynomial is also a homogeneous subspace provided that the grading is nontrivial. In the case of the trivial grading, if the field contains at least 3 elements then the image of every multilinear polynomial is a vector subspace. Then we proceed with the Jordan algebra structure $UJ_n$ obtained from $UT_n$ by the Jordan (symmetric) product $a\circ b= ab+ba$ provided that the characteristic of the base field is different from 2. The description of all group gradings on $UJ_n$ is more complicated than that on $UT_n$, see \cite{KYa}, there appear gradings that are not isomorphic to elementary ones. The gradings on $UJ_2$ were described in \cite{KMa}. We consider each one of these gradings, and prove that the image of a multilinear graded  polynomial is always a homogeneous subspace. No restrictions on the base field are imposed (apart from the characteristic being different from 2). An analogous result is deduced for the Lie algebra $UT_2^{(-)}$ obtained from $UT_n$ by substituting the associative product by the Lie bracket $[a,b]=ab-ba$. Finally we consider $UJ_3$ equipped with the natural $\mathbb{Z}_3$-grading: $\deg e_{ij} = j-i\pmod{3}$ for every $i\le j$, assuming the base field infinite and of characteristic different from 2. We prove that the image of a multilinear graded polynomial is always a homogeneous subspace. 

We hope that this paper will initiate a more detailed study of images of polynomials on algebras with additional structures.

\section{Preliminaries}

Unless otherwise stated, we denote by $F$ an arbitrary field and $\mathcal{A}$ an associative algebra over $F$. Given a group $G$, a $G$-grading on $\mathcal{A}$ is a decomposition of $\mathcal{A}$ in a direct sum of subspaces $\mathcal{A}=\bigoplus_{g\in G}\mathcal{A}_{g}$ such that $\mathcal{A}_{g}\mathcal{A}_{h}\subset \mathcal{A}_{gh}$, for all $g$, $h \in G$. We define the support of a $G$-grading on $\mathcal{A}$ as the subset $supp(\mathcal{A})=\{g\in G\mid \mathcal{A}_{g}\neq0\}$. A subspace $U$ of $\mathcal{A}$ is called homogeneous if $U=\bigoplus_{g\in G}(U\cap \mathcal{A}_{g})$. A graded homomorphism between two graded algebras $\mathcal{A}=\bigoplus_{g\in G}\mathcal{A}_{g}$ and $\mathcal{B}=\bigoplus_{g\in G}\mathcal{B}_{g}$ is defined as an algebra homomorphism $\varphi\colon \mathcal{A}\rightarrow \mathcal{B}$ such that $\varphi(\mathcal{A}_{g})\subset \mathcal{B}_{g}$ for every $g\in G$. We denote by $F\langle X \rangle^{gr}$ the free $G$-graded associative algebra generated by a set of noncommuting variables $X=\{x_{i}^{(g)}\mid i\in\mathbb{N},g\in G\}$. We also denote the neutral (that is of degree $1\in G$) variables by $y$ and call them even variables, and the non neutral ones by $z$ and  we call them odd variables. We draw the reader's attention that odd variables may have different degrees in the $G$-grading. 

We define the image of a graded polynomial on an algebra as in \cite{Kul}.

\begin{defi}
Let $f\in F\langle X \rangle^{gr}$ be a $G$-graded polynomial. The image of $f$ on the $G$-graded algebra $\mathcal{A}$ is the set
\[
Im(f)=\{a\in\mathcal{A}\mid a=\varphi(f) \ \mbox{for some graded homomorphism} \ \varphi\colon F\langle X \rangle^{gr}\rightarrow \mathcal{A}\}
\]
\end{defi}

Equivalently, if $f(x_{1}^{(g_{1})},\dots,x_{n}^{(g_{n})})\in F\langle X \rangle^{gr}$, then the image of $f$ on the algebra $\mathcal{A}$ is the set $Im(f)=\{f(a_{1}^{(g_{1})},\dots,a_{n}^{(g_{n})})\mid a_{i}^{(g_{i})}\in \mathcal{A}_{g_{i}}\}$. We will also denote the image of $f$ on $\mathcal{A}$ by $f(\mathcal{A})$.

We now recall some basic properties of images of graded polynomials on algebras that will be used throughout the paper. 

\begin{prp}\label{basicprop}
Let $f\in F\langle X \rangle^{gr}$ be a polynomial and $\mathcal{A}$ a $G$-graded algebra. 
\begin{enumerate}
    \item Let $U$ be one-dimensional subspace of $\mathcal{A}$ such that $Im(f)\subset U$ and assume that $\lambda Im(f)\subset Im(f)$ for every $\lambda \in F$. Then either $Im(f)=\{0\}$ or $Im(f)=U$;
    \item If $1\in \mathcal{A}$ and $f\in F\langle X \rangle^{gr}$ is a multilinear polynomial in neutral variables such that the sum of its coefficients is nonzero, then $Im(f)=\mathcal{A}_{1}$;
    \item $Im(f)$ is invariant under graded endomorphisms of $F\langle X \rangle^{gr}$;
    \item If $supp(\mathcal{A})$ is abelian and $f\in F\langle X \rangle^{gr}$ is multilinear, then $Im(f)$ is a homogeneous subset.
    \end{enumerate}
\end{prp}

\begin{proof}
The proofs of the first and third items are straightforward. For the second item it is enough to recall that if $\mathcal{A}$ is a graded algebra with $1$, then $1\in\mathcal{A}_{1}$. Hence, given $a\in\mathcal{A}_{1}$ we have $a=f(\alpha^{-1} a,1,\dots,1)$, where $\alpha\neq0$ is the sum of the coefficients of $f$, and then $Im(f)=\mathcal{A}_{1}$. For the last item, let $g_{1}$, \dots, $g_{m}$ be the homogeneous degree of the variables that occur in $f$. If some $g_{i}\notin supp(\mathcal{A})$, then $Im(f)=\{0\}$ is a homogeneous subspace. Otherwise, since $supp(\mathcal{A})$ is abelian, we have that each monomial of $f$ is of homogeneous degree $g_{1}\cdots g_{m}$, and hence the same holds for $f$.  
\end{proof}

We say that a nonzero polynomial $f\in F\langle X \rangle^{gr}$ is a graded polynomial identity for a $G$-graded algebra $\mathcal{A}$ if its image on $\mathcal{A}$ is zero. The set of all graded polynomial identities of $\mathcal{A}$ will be denoted by $Id^{gr}(\mathcal{A})$. It is easy to check that $Id^{gr}(\mathcal{A})$ is actually an ideal of $F\langle X \rangle^{gr}$ invariant under graded endomorphisms of $F\langle X \rangle^{gr}$. It is called the $T_{G}$-ideal of $\mathcal{A}$. Given a nonempty subset $S$ of $F\langle X \rangle^{gr}$, we denote by $\langle S \rangle^{T_{G}}$ the $T_{G}$-ideal generated by $S$, that is the least $T_{G}$-ideal that contains the set $S$. The linearisation process also holds for graded polynomials, and as in the ordinary case we have the following statement.

\begin{prp}\label{multiidentity}
If $\mathcal{A}$ satisfies a graded polynomial identity, then $\mathcal{A}$ also satisfies a multilinear one. Moreover, if $char(F)=0$, then $Id^{gr}(\mathcal{A})$ is generated by its multilinear polynomials.
\end{prp}

Let now $\mathcal{A}=UT_{n}$ be the algebra of $n\times n$ upper triangular matrices over the field $F$. A $G$-grading on $\mathcal{A}$ is said to be elementary if all elementary matrices are homogeneous in this grading, or equivalently, if there exists an $n$-tuple $(g_{1},\dots,g_{n})\in G^{n}$ such that $\deg(e_{ij})=g_{i}^{-1}g_{j}$. A theorem of Valenti and Zaicev states that every grading on $UT_{n}$ is essentially elementary.

\begin{thm}[\cite{VZa}]\label{gradingsupper}
Let $G$ be a group and let $F$ be a field. Assume that $UT_{n}=\mathcal{A}=\bigoplus_{g\in G} \mathcal{A}_{g}$ is $G$-graded. Then $\mathcal{A}$ is $G$-graded isomorphic to $UT_{n}$ endowed with some elementary $G$-grading.
\end{thm}

By Proposition 1.6 from \cite{VKV} we have that an elementary grading on $UT_{n}$ is completely determined by the sequence $(\deg(e_{12}),\deg(e_{23}),\ldots,\deg(e_{n-1,n}))\in G^{n-1}$.

We recall now some recent results about the description of images of multilinear polynomials on the algebra of upper triangular matrices. We start with the definition of the so-called commutator degree of an associative polynomial.

\begin{defi}
Let $f\in F\langle X \rangle$ be a polynomial. We say that $f$ has commutator degree $r$ if 
\[
f\in\langle [x_{1},x_{2}]\cdots [x_{2r-1},x_{2r}]\rangle^{T} \ \mbox{and} \ f\notin\langle [x_{1},x_{2}]\cdots [x_{2r+1},x_{2r+2}]\rangle^{T}.
\]
\end{defi}
In \cite{GMe}, Gargate and de Mello used the above definition to give a complete description of images of multilinear polynomials on $UT_{n}$ over infinite fields. Denoting by $J$ the Jacobson radical of $UT_{n}$ and  $J^{0}=UT_{n}$, they proved the following theorem.

\begin{thm}\label{tGargateThiago}
Let $F$ be an infinite field and let $f\in F\langle X \rangle$ be a multilinear polynomial. Then $Im(f)$ on $UT_{n}$ is $J^{r}$ if and only if $f$ has commutator degree $r$.
\end{thm}

One of the main steps in the proof of Theorem \ref{tGargateThiago} was the characterization the polynomials of commutator degree $r$ in terms of their coefficients. An instance of such characterization has already been known, see \cite{GMe} Lemma 3.3(2). 

\begin{lem}[\cite{GMe}]\label{sumofcoeffi}
Let $F$ be an arbitrary field and let $f\in F\langle X \rangle$ be a multilinear polynomial. Then $f\in \langle [x_{1},x_{2}]\rangle^{T}$ if and only if the sum of its coefficients is zero.
\end{lem}

It is worth mentioning that the above theorem has been extended for a larger class of fields by Luo and Wang in \cite{LWa}.

\begin{thm}[\cite{LWa}]\label{TLuoWang}
Let $n\geq 2 $ be an integer, let $F$ be a field with at least $n(n-1)/2$ elements and let $f\in F\langle X \rangle$ be a multilinear polynomial. If $f$ has commutator degree $r$, then $Im(f)$ on $UT_{n}$ is $J^{r}$.
\end{thm}

In the next corollary we denote by $UT_{n}^{(-)}$ the Lie algebra defined on $UT_{n}$ by means of the Lie bracket $[a,b]=ab-ba$.

\begin{cor}
Let $F$ be a field with at least $n(n-1)/2$ elements and let $f\in L(X)$ be a multilinear Lie polynomial. Then $Im(f)$ on $UT_{n}^{(-)}$ is $J^{r}$, for some $0\leq r \leq n$.
\end{cor}

\begin{proof}
We use the Poincaré-Birkhoff-Witt Theorem (and more precisely the Witt Theorem) to consider the free Lie algebra $L(X)$ as the subalgebra of $F\langle X \rangle^{(-)}$ generated by the set $X$. Since $F\langle X\rangle$ is the universal enveloping algebra of $L(X)$, given a multilinear Lie polynomial $f\in L(X)$ there exists an associative polynomial $\tilde{f}\in F\langle X \rangle$ such that $Im(f)$ on $UT_{n}^{(-)}$ is equal to $Im(\tilde{f})$ on $UT_{n}$. Now it is enough to apply Theorem \ref{TLuoWang}.
\end{proof}

Let $UT_{n}(d_{1},\ldots,d_{k})$ be the upper block-triangular matrix algebra, that is, the subalgebra of $M_{n}(F)$ consisting of all block-triangular matrices of the form
\[
\begin{pmatrix}
     A_{1}& & *  \\
     & \ddots & \\
     0 & & A_{k} 
\end{pmatrix}
\]
where $n=d_{1}+\cdots+d_{k}$ and  $A_{i}$ is a $d_{i}\times d_{i}$ matrix. We will denote by $T$ the subalgebra of $UT_{n}(d_{1},\ldots,d_{k})$ which consists of only triangular blocks of sizes $d_{i}$ on the main diagonal and zero elsewhere. That is, 
\[
T=\begin{pmatrix}
UT_{d_{1}} & &0 \\
 & \ddots & \\
 0& & UT_{d_{k}}
\end{pmatrix}
\]
As a consequence of the above theorem we obtain the following lemma.

\begin{lem}\label{lblock}
Let $F$ be a field with at least $n(n-1)/2$ elements and let $f\in F\langle X \rangle$ be a multilinear polynomial of commutator degree $r$. Then the image $Im(f)$ on $T$ is $J^{r}$, where $J=Jac(T)$ is the Jacobson radical of $T$.
\end{lem}

\begin{proof}
We note that $T\cong UT_{d_{1}}\times\cdots\times UT_{d_{k}}$. Hence, by \cite[Proposition 5.60]{Bre},
\[
J=\begin{pmatrix}
    J_{d_{1}} & & \\
     & \ddots & \\
     & & J_{d_{k}}
\end{pmatrix}
\]
where $J_{d_{i}}=Jac(UT_{d_{i}})$. Therefore by Theorem \ref{TLuoWang}, we have
\[
f(T)=\begin{pmatrix}
    f(UT_{d_{1}}) & &  \\
     &\ddots &\\
     & & f(UT_{d_{k}})
\end{pmatrix} =\begin{pmatrix}
    J_{d_{1}}^{r} & &  \\
     &\ddots &\\
     & & J_{d_{k}}^{r}
\end{pmatrix}=J^{r}.
\qedhere
\]
\end{proof}

Throughout this paper we use the letters $w_{i}$ and $w_{i}^{(j)}$ to denote commuting variables. We recall the following well known result about commutative polynomials.

\begin{lem}\label{lcomutpoly}
Let $F$ be an infinite field and let $f_{1}(w_{1},\dots,w_{m})$, \dots,  $f_{n}(w_{1},\dots,w_{m})$ be commutative polynomials. Then there exist $a_{1}$, \dots,  $a_{m}\in F$ such that 
\[
f_{1}(a_{1},\dots,a_{m})\neq0,\quad \dots, \quad f_{n}(a_{1},\dots,a_{m})\neq0.
\]
\end{lem}

A similar result also holds for finite fields, as long as some boundedness  on the degrees of the variables is given (see \cite[Proposition 4.2.3]{Dre}).

\begin{lem}\label{lcomutpolyfinite}
Let $F$ be a finite field with $n$ elements and let $f=f(w_{1},\dots,w_{m})$ be a nonzero polynomial. If $\deg_{w_{i}}(f)\leq n-1$ for every $i=1$, \dots, $n$, then there exist $a_{1}$, \dots, $a_{m}\in F$ such that $f(a_{1},\dots,a_{m})\neq0$.
\end{lem}

\begin{cor}\label{ccomutpolyfinite}
Let $F$ be a finite field with $n$ elements and let $f_{1}(w_{1}\dots,w_{m})$, \dots, $f_{n-1}(w_{1},\dots,w_{m})$ be nonzero polynomials in commuting variables. If $\deg_{w_{i}}(f_{j})\leq 1$ for all $i$ and $j$, then there exist $a_{1}$,  \dots, $a_{m}\in F$ such that
\[
f_{1}(a_{1},\dots,a_{m})\neq0,\quad \dots,\quad f_{n-1}(a_{1},\dots,a_{m})\neq0.
\]
\end{cor}


\section{Graded central polynomials for $UT_{n}$}
\label{sect3}

Our goal in this section is to prove the non existence of graded central polynomials for the graded algebra of upper triangular matrices with entries in an arbitrary field. It is well known that the algebra of upper block triangular matrices has no central polynomials, see \cite[Lemma 1]{gz_ijm}.

We will denote by $Z(\mathcal{A})$ the centre of the algebra $\mathcal{A}$.  

\begin{defi}
Let $f\in F\langle X \rangle^{gr}$. We say that $f$ is a graded central polynomial for the algebra $\mathcal{A}$ if $Im(f)\subset Z(\mathcal{A})$ and $f\notin Id^{gr}(\mathcal{A})$.
\end{defi}

We recall the following fact from \cite[Lemma 1.4]{VKV}.

\begin{lem}
Let $UT_{n}$ be endowed with some elementary grading. Then the subspace of all diagonal matrices is homogeneous of neutral degree.
\end{lem}

\begin{thm}
Let $UT_{n}=\mathcal{A}=\bigoplus_{g\in G}\mathcal{A}_{g}$ be a $G$-grading on the algebra of upper triangular matrices over an arbitrary field. If $n>1$ then there exist no graded central polynomials for $\mathcal{A}$.
\end{thm}

\begin{proof}
By Theorem \ref{gradingsupper} we have that $\mathcal{A}$ is graded isomorphic to some elementary grading on $UT_{n}$. Hence we may reduce our problem to elementary gradings. Now we assume that $f\in F\langle X \rangle^{gr}$ is a polynomial with zero constant term, such that $Im(f)$ on $\mathcal{A}$ is contained in $F=Z(\mathcal{A})$. We write $f$ as $f=f_{1}+f_{2}$ where $f_{1}$ contains neutral variables only and $f_{2}$ has at least one non neutral variable in each of its monomials. Consider $a_{1}$, \dots, $a_{m}\in\mathcal{A}_{1}$, and $b_{1}$, \dots, $b_{l}$ non neutral variables (of homogeneous degree $\ne 1$) that occur in $f$. Hence $f(a_{1},\dots,a_{m},b_{1},\dots,b_{l})=f_{1}(\overline{a}_{1},\dots,\overline{a}_{m})+j_{1}+j_{2}$ where $j_{1}$, $j_{2}\in J$, the Jacobson radical of $\mathcal{A}$, and $\overline{a}_{i}$ is the diagonal part of $a_{i}$. Since $Im(f)\subset F$, then $j_{1}+j_{2}=0$ and hence $Im(f)=Im(f_{1})$, where the image of $f_{1}$ is taken on diagonal matrices only. Now, note that if $\lambda_{1}$, \dots, $\lambda_{m}\in F$ are arbitrary, then 
\[
f_{1}(\lambda_{1}e_{11},\dots,\lambda_{m}e_{11})=f_{1}(\lambda_{1},\dots,\lambda_{m})e_{11}.
\]
Since $Im(f_{1})\subset F$, we must have $f_{1}(\lambda_{1},\dots,\lambda_{m})=0$. Hence, for diagonal matrices $D_{i}=\displaystyle\sum_{k=1}^{n}\lambda_{k}^{(i)}e_{kk}$ we have 
\[
f_{1}(D_{1},\dots,D_{m})=\sum_{k=1}^{n}f_{1}(\lambda_{1}^{(k)},\dots,\lambda_{k}^{(m)})e_{kk}=0,
\]
and thus $Im(f)=\{0\}$. We conclude the non existence of graded central polynomials for $UT_{n}$.
\end{proof}

\section{Certain $\Z_{q}$-gradings on $UT_{n}$}
\label{sect4}

Throughout this section we denote $UT_{n}=\mathcal{A}$, endowed with the elementary $\mathbb{Z}_{q}$-grading given by the following sequence in $\mathbb{Z}_{q}^{n}$
\[
(\overline{0},\overline{1},\dots,\overline{q-2},\underbrace{\overline{q-1},\overline{q-1},\dots,\overline{q-1}}_{\text{$n-q+1$ times}}). \]
Given $q\leq n$ an integer, we study the images of multilinear graded polynomials on $\mathcal{A}$.

One can see that for $q=n$ we have the natural $\Z_{n}-$grading on $UT_{n}$ given by $\deg e_{ij}=j-i\pmod{n}$ for every $i\le j$.

We note that the neutral component of $UT_{n}$ is given by a block triangular matrix with $q-1$ triangular blocks of size one each and a triangular block of size $n-q+1$ in the bottom right corner
\[
\mathcal{A}_{0}=\begin{pmatrix}
    * &   &   & 0  \\
      & \ddots & & \\
      & & * & \\
      0& & & UT_{n-q+1}
\end{pmatrix}
\]
For $l\in\{1,\dots,q-1\}$ we have that the homogeneous component of degree $\overline{l}$ is given by
\[
\mathcal{A}_{\overline{l}}=span\{e_{i,i+l}, e_{q-l,j} \mid i=1,\dots,q-l, j=q+1,\dots,n\}.
\]
For $1\leq r \leq n-q$ we also define the following homogeneous subspaces of $A_{\overline{l}}$
\[
\mathcal{B}_{\overline{l},r}=span\{e_{q-l,j} \mid j=q+r,\ldots,n\}.
\]
An immediate computation shows that the following are graded identities for $\mathcal{A}$
\begin{align}
&[y_{1},y_{2}]z\equiv 0 \label{identity1} \\
&z_{1}z_{2}\equiv 0  \label{identity2} \\
&{[y_{1},y_{2}]}\cdots [y_{2(n-q+1)-1},y_{2(n-q+1)}]\equiv 0 \label{identity3}
\end{align}
where the variables $y_i$ are neutral ones, $z$, $z_{1}$, $z_{2}$ are non neutral variables and $\deg(z_{1})+\deg(z_{2})=\overline{0}$. A complete description of the graded polynomial identities for elementary gradings on $UT_{n}$ was given in \cite{VKV} for infinite fields and in \cite{GRi} for finite fields.

We state several lemmas concerning the description of some graded polynomials on $\mathcal{A}$. In the upcoming lemmas, unless otherwise stated, we assume that the field $F$ has at least $n(n-1)/2$ elements and $f\in F\langle X \rangle^{gr}$ is a multilinear polynomial.

\begin{lem}\label{l1Zq}
If $f=f(y_{1},\dots,y_{m})$, then $Im(f)$ on $\mathcal{A}$ is a homogeneous vector subspace.
\end{lem}

\begin{proof}
It is enough to apply Lemma \ref{lblock}.
\end{proof}

In the next two lemmas we will assume that $f=f(z_{1},\dots,z_{l},y_{l+1},\dots,y_{m})$ where $\deg(z_{i})=\overline{1}$, $1\leq i \leq l$. It is obvious that in this case one must have $Im(f)$ on $\mathcal{A}$ as a subset of $\mathcal{A}_{\overline{l}}$. Modulo the identity $(1)$ we rewrite the polynomial $f$ as
\[
f=\sum_{\bm{i_{1}},\dots, \bm{i_{l}}}y_{\bm{i_{1}}}z_{1}y_{\bm{i_{2}}}z_{2}\cdots y_{\bm{i_{l}}}z_{l}g_{\bm{i_{1}},\dots,\bm{i_{l}}}+h
\]
where $y_{\bm{i_{j}}}=y_{i_{j_{1}}}\cdots y_{i_{j_{k_{j}}}}$ is such that $i_{j_{1}}<\cdots <i_{j_{k_{j}}}$. Moreover $g_{\bm{i_{1}},\dots,\bm{i_{l}}}$ is the polynomial obtained by permuting the neutral variables whose indices are different from either of $\bm{i_{1}}$, \dots, $\bm{i_{l}}$, and forming a linear combination of such monomials. Furthermore, $h$ is the sum of  polynomial that differ from the first summand of $f$ by nontrivial permutations of the odd variables.

Among all polynomials $g_{\bm{i_{1}},\dots,\bm{i_{l}}}$ (including those in $h$), we choose one of least commutator degree, say $g$, of commutator degree $r$. Up to permuting the odd variables, we can assume that the polynomial $g$ occurs in the first summand of $f$. 

Hence, in case $1\leq r \leq n-q$, we can improve the inclusion $Im(f)\subset \mathcal{A}_{\overline{l}}$ to $Im(f)\subset \mathcal{B}_{\overline{l},r}$. Our goal is to prove that $Im(f)=\mathcal{A}_{\overline{l}}$ in case $r=0$ and $Im(f)=\mathcal{B}_{\overline{l},r}$ otherwise.

\begin{lem}\label{l2Zq}
If $1\leq r \leq n-q$, then $Im(f)=\mathcal{B}_{\overline{l},r}$.
\end{lem}

\begin{proof}
Let $\overline{g}=y_{\bm{i_{1}}}z_{1}y_{\bm{i_{2}}}z_{2}\cdots y_{\bm{i_{l}}}z_{l}g$ be a nonzero summand of $f$ written as above, where the commutator degree of $g$ is $r$. We consider the following evaluations: the variables in $y_{\bm{i_{1}}}$ by $e_{q-l,q-l}$, the ones in $y_{\bm{i_{2}}}$ by $e_{q-l+1,q-l+1}$,\dots, and all variables in $y_{\bm{i_{l}}}$ by $e_{q-1,q-1}$. We also put $z_{1}=e_{q-l,q-l+1}$, $z_{2}=e_{q-l+1,q-l+2}$, \dots, $z_{l-1}=e_{q-2,q-1}$, and $z_{l}=\sum_{k=q}^{n}w_{k}e_{q-1,k}$. Since $g$ is of commutator degree $r$, Theorem \ref{TLuoWang} enables us to evaluate the even variables in $g$ by matrices from 
\[
\begin{pmatrix}
0&0\\
0&UT_{n-q+1}
\end{pmatrix}
\]
in order to obtain the matrix $e_{q,q+r}+e_{q+1,q+r+1}+\cdots+ e_{n-r,n}$.

Note that the evaluations that we have considered allow us to reduce the study of the image of $f$ to the polynomial $\overline{g}$. Under these evaluations we have 
\begin{align*}
\overline{g}&=(w_{q}e_{q-l,q}+w_{q+1}e_{q-l,q+1}+\cdots+w_{n}e_{q-l,n})(e_{q,q+r}+e_{q+1,q+r+1}+\cdots+e_{n-r,n})\\
&=w_{q}e_{q-l,q+r}+w_{q+1}e_{q-l,q+r+1}+\cdots+w_{n-r}e_{q-l,n}.
\end{align*}

Taking a matrix $B\in \mathcal{B}_{\overline{l},r}$, say $B=b_{q}e_{q-l,q+r}+\cdots+b_{n-r}e_{q-l,n}$ we can easily realize $B$ as image of $\overline{g}$ by choosing $w_{j}=b_{j}$, $j=q$, \dots, $n-r$. 

Hence $f(\mathcal{A})=\mathcal{B}_{\overline{l},r}$.
\end{proof}

\begin{lem}\label{l3Zq}
If $F$ is a field with at least $n$ elements and $r=0$, then  $Im(f)=\mathcal{A}_{\overline{l}}$.
\end{lem}

\begin{proof}
Denoting by $\mathcal{D}$ the homogeneous subspace of diagonal matrices of $\mathcal{A}$, we consider the following homogeneous subalgebra of $\mathcal{A}$:
\[
\mathcal{S}=\mathcal{D}\oplus \bigoplus_{1\leq l\leq q-1}\mathcal{A}_{\overline{l}}.
\]
We will show that $Im(f)$ on $\mathcal{S}$ is $\mathcal{A}_{\overline{l}}$ which is enough to conclude the lemma. Note that $\mathcal{S}$ still satisfies the identity (\ref{identity2}) and it also satisfies $[y_{1},y_{2}]\equiv0$.

By the identity $[y_{1},y_{2}]\equiv0$, we may write the polynomial $\overline{g}$ as 
\[
\beta y_{\bm{i_{1}}}z_{1}y_{\bm{i_{2}}}z_{2}\cdots y_{\bm{i_{l}}}z_{l}y_{\bm{i_{l+1}}}
\]
where $\beta$ is the sum of all coefficients of the polynomial $g$ and $y_{\bm{i_{l+1}}}$ is the product of the variables of $g$ in increasing order of the indices. We claim that $\beta\neq 0$. To this end it is enough to check that $\overline{g}$ is not a consequence of a commutator. In fact,  if $\overline{g}$ is a consequence of a commutator it must be in the variables of $g$ since the remaining variables of $\overline{g}$ are given in a fixed order in all of its monomials. However $g$ has commutator degree zero, which proves our claim.  

Now we write $f=f(z_{1},\dots,z_{l},y_{l+1},\dots,y_{m})$ as 
\[
f=\sum_{j=1}^{l}f_{j}
\]
where $f_{j}$ is the sum of all monomials of $f$ such that the variable $z_{l}$ is in the $j$-th position in relation to the odd variables.

For each $j=1$, \dots, $l$, we write
\[
f_{j}=\sum_{\sigma\in S_{l}^{(j)}}f_{j,\sigma}
\]
where $S_{l}^{(j)}=\{\sigma\in S_{l}|\sigma(l)=j\}$, $f_{j,\sigma}$ is the sum of all monomials of $f_{j}$ where the order of the odd variables is given by the permutation $\sigma$. 

Taking $\displaystyle z_{i}=\sum_{k=1}^{q-1}w_{k}^{(i)}e_{k,k+1}+w_{q}^{(i)}e_{q-1,q+1}+\cdots+w_{n-1}^{(i)}e_{q-1,n}$ and $\displaystyle y_{j}=\sum_{k=1}^{n}w_{k}^{(j)}e_{kk}$ we have
\begin{align*}
&f_{l,id}(z_{1},\dots,z_{l},y_{l+1},\dots,y_{m})=\sum_{k=1}^{q-l}p_{k}w_{k}^{(1)}w_{k+1}^{(2)}\cdots w_{k+l-1}^{(l)}e_{k,k+l}\\
&+ p_{q-l+1}w_{q-l}^{(1)}\cdots w_{q-2}^{(l-1)}w_{q}^{(l)}e_{q-l,q+1}+\cdots+ p_{n-l}w_{q-l}^{(1)}\cdots w_{q-2}^{(l-1)}w_{n-1}^{(l)}e_{q-l,n}
\end{align*}
where $p_{k}$, $k=1$, \dots, $n-l$, are polynomials in the variables $w^{(l+1)}$, \dots, $w^{(m)}$.
We note that all polynomials $p_{k}$, $k=1$, \dots, $n-l$, are nonzero ones. Indeed, we just have to check that different monomials in $f_{l,id}$ give  different monomials in $p_{k}$. To this end, note that if $m_{1}$ and $m_{2}$ are different monomials in $f_{l,id}$, then there exists some even variable $y_{j}$ such that the quantity of preceding odd variables in relation to $y_{j}$ is distinct in $m_{1}$ and $m_{2}$. This gives us variables $w^{(j)}$ with different lower indices in the two monomials in $p_{k}$ given by $m_{1}$ and $m_{2}$, which proves our claim. Moreover, we note that every variable in each monomial of the polynomial $p_{k}$ appears exactly once.  

Since we have at most $n-1$ polynomials $p_{k}$, by Corollary \ref{ccomutpolyfinite} there exist evaluations of the even variables $y_{j}$ by diagonal matrices $D_{j}$ such that $p_{k}$ take nonzero values simultaneously for all $k$. In case $F$ is infinite the same conclusion holds by applying Lemma \ref{lcomutpoly}.
Hence 
\begin{align*}
f_{l}&=\sum_{k=1}^{q-l}\bigg(\sum_{\sigma\in S_{l}^{(l)}}\alpha_{\sigma}w_{k}^{(\sigma(1))}\cdots w_{k+l-2}^{(\sigma(l-1))}\bigg)w_{k+l-1}^{(l)}e_{k,k+l}\\ 
&+ \bigg(\sum_{\sigma\in S_{l}^{(l)}}\alpha_{\sigma}w_{q-l}^{(\sigma(1))}\cdots w_{q-2}^{(\sigma(l-1))}\bigg)w_{q}^{(l)}e_{q-l,q+1}\\
&+\cdots+ \bigg(\sum_{\sigma\in S_{l}^{(l)}}\alpha_{\sigma}w_{q-l}^{(\sigma(1))}\cdots w_{q-2}^{(\sigma(l-1))}\bigg)w_{n-1}^{(l)}e_{q-l,n}
\end{align*}
with $\alpha_{id}\neq0$. So the polynomials inside the brackets above are nonzero ones and each of their monomials have variables of degree one. Applying Corollary \ref{ccomutpolyfinite} once again we may evaluate the variables $z_{1}$, \dots, $z_{l-1}$ by matrices in $C_{1}$, \dots, $C_{l-1}\in A_{\overline{1}}$ such that all these polynomials take nonzero values on $F$ (in case $F$ is infinite we apply Lemma \ref{lcomutpoly}). Denote by $\alpha_{l,k}\in F\setminus\{0\}$, $k =1$, \dots, $q-l$, the values of the polynomials inside the brackets after such evaluations. 

Therefore
\begin{align*}
&f(C_{1},\dots,C_{l-1},z_{l},D_{l+1},\dots,D_{m})\\
&=\sum_{k=1}^{q-l}\bigg(\alpha_{1,k}w_{k}^{(l)}+\cdots +\alpha_{l-1,k}w_{k+l-2}^{(l)} +\alpha_{l,k}w_{k+l-1}^{(l)}\bigg)e_{k,k+l}\\
&+\bigg(\alpha_{1,q}w_{q-l}^{(l)}+\cdots + \alpha_{l-1,q}w_{q-2}^{(l)}+\alpha_{l,q-l}w_{q}^{(l)}\bigg)e_{q-l,q+1}\\
&+\cdots+ \bigg(\alpha_{1,n-1}w_{q-l}^{(l)}+\cdots + \alpha_{l-1,n-1}w_{q-2}^{(l)}+\alpha_{l,q-l}w_{n-1}^{(l)}\bigg)e_{q-l,n}
\end{align*}
where $\alpha_{l,k}\neq0$ for every $k=1$, \dots, $q-l$.

Then given a matrix $B=\displaystyle\sum_{k=1}^{q-l}b_{k}e_{k,k+l}+b_{q-l+1}e_{q-l,q+1}+\cdots+ b_{n-l}e_{q-l,n}\in \mathcal{A}_{\overline{l}}$ we take 
\[
f(C_{1},\dots,C_{l-1},z_{l},D_{l+1},\dots,D_{m})=B
\]
and we obtain a linear system in the variables $w^{(l)}$ whose solution (not necessarily unique) can be found recursively. 
\end{proof}

\begin{thm}\label{maintheorem}
Let $F$ be a field with at least $n(n-1)/2$ elements, let $UT_{n}=\bigoplus_{k\in\mathbb{Z}_{q}}\mathcal{A}_{k}$ be endowed with the elementary $\mathbb{Z}_{q}$-grading given by the sequence $(\overline{0},\overline{1},\dots,\overline{q-2},\overline{q-1},\dots,\overline{q-1})$ and let $f\in F\langle X \rangle^{gr}$ be a multilinear polynomial. Then $Im(f)$ on $UT_{n}$ is $\{0\}$, $J^{r}$, $\mathcal{B}_{\overline{l},r}$ or $\mathcal{A}_{\overline{l}}$, where $J=Jac(\mathcal{A}_{0})$. In particular, the image is always a homogeneous vector subspace.
\end{thm}

\begin{proof}
By Lemmas \ref{l1Zq}, \ref{l2Zq} and \ref{l3Zq} we only need to analyse the case where $f$ is a polynomial in neutral variables and non neutral ones (that is of degree different from $\overline{1}$). However we can reduce this case to the aforementioned lemmas. Indeed, modulo the graded identities (\ref{identity1}), (\ref{identity2}), (\ref{identity3}), let $f$ and $g$ be as in the comments before Lemma \ref{l2Zq}  and let $r$ be the commutator degree of $g$. Hence $Im(f)\subset \mathcal{B}_{\overline{l},r}$ if $r\neq0$ and $Im(f)\subset \mathcal{A}_{\overline{l}}$ otherwise. By Proposition \ref{basicprop}(3), the image of the polynomial $\tilde{f}$ obtained from $f$ by evaluating every non neutral variable $z_{i}$ of homogeneous degree $\overline{k}$ by a product of $k$ variables of homogeneous degree $\overline{1}$, is contained in $Im(f)$. But the polynomial $g$ defined for $f$ (see the comments before Lemma \ref{l2Zq}) is the same as the one defined for $\tilde{f}$. This allows us to get $\mathcal{B}_{\overline{l}}^{r}\subset Im(\tilde{f})$ in case $r\neq0$ and $\mathcal{A}_{\overline{l}}\subset Im(\tilde{f})$ otherwise.
\end{proof}

\begin{rem}
Considering similar computations one can easily see that the result is also valid for the elementary grading defined by the identities $z[y_{1},y_{2}]\equiv0$, (\ref{identity2}), and (\ref{identity3}). 
\end{rem}

In the next corollary we are assuming that $F$ is a field of characteristic zero and $\mathcal{A}=UT_{n}$ is endowed with the elementary $\Z_{q}$-grading given by the sequence $(\overline{0},\overline{1},\dots,\overline{q-2},\overline{q-1},\dots,\overline{q-1})$.  

\begin{cor}
The $T_{G}$-ideal $Id^{gr}(\mathcal{A})$ is generated by the graded identities (\ref{identity1}), (\ref{identity2}), and (\ref{identity3}). 
\end{cor}

\begin{proof}
Let $f\in Id^{gr}(\mathcal{A})$. By Proposition \ref{multiidentity} we may assume that $f$ is multilinear. Note that in the proof of Theorem \ref{maintheorem} and in the lemmas that precede it, we have shown that if $f$ is not a consequence of the identities (\ref{identity1}), (\ref{identity2}), and (\ref{identity3}), then $Im(f)\neq\{0\}$. In other words, if $f\in Id^{gr}(\mathcal{A})$, then $f$ is a consequence of the aforementioned identities.  
\end{proof}

We recall that in case $q=n$ we have the natural $\Z_{n}$-grading on $UT_{n}$. 

\begin{cor}\label{cnaturalUTn}
Let $F$ be a field with at least $n$ elements, let $UT_{n}$ be endowed with the natural $\mathbb{Z}_{n}$-grading and let $f\in F\langle X \rangle^{gr}$ be a multilinear polynomial. Then the image of $f$ on $UT_{n}$ is either zero or some homogeneous component.
\end{cor}

\begin{proof}
It follows from Proposition \ref{basicprop}(2), Lemma \ref{sumofcoeffi}, and from the proof of Lemma \ref{l3Zq}.
\end{proof}

\begin{rem}
Note that the same result also holds if we consider the natural $\Z$-grading on $UT_{n}$. Analogous results hold for the lower triangular matrix algebra $LT_{n}$ as well (we will use this remark in the next section).
\end{rem}

\section{Graded identities and traceless matrices}
\label{sect5}

In this section we give a sufficient condition for the subspace of the traceless matrices to be contained in the image of a multilinear polynomial on the full matrix algebra.

We start by recalling the following result from \cite{ARo}.

\begin{thm}[\cite{ARo}]\label{tAR}
Let $D$ be a division ring, $n \geq 2$ an integer, and $A \in M_n(D)$ a non central matrix. Then $A$ is similar (conjugate) to a matrix in $M_n(D)$ with at most one non-zero entry on the main diagonal. In particular, if $A$ has trace zero, then it is similar to a matrix in $M_n(D)$ with only zeros on the main diagonal.
\end{thm}

Consider the natural $\mathbb{Z}$-grading on $M_{n}(F)=\bigoplus_{r\in\mathbb{Z}}(M_{n}(F))_{r}$ given by
\[
(M_{n}(F))_{r}=\left\{\begin{array}{cl}
    span\{e_{k,k+r}\mid k=1,\dots,n-r\}, &  \mbox{if} \ 0\leq r \leq n-1 \\
    span\{e_{k-r,k}\mid k=1,\dots,n+r\}, &  \mbox{if} \ -n+1\leq r\leq -1
    \\
    \{0\}, & \mbox{elsewhere}
\end{array}\right.
\]

In this section we denote $F\langle X \rangle^{gr}$ the free $\mathbb{Z}$-graded algebra. We also keep the notation of using variables $y$'s for neutral variables and $z$'s for non neutral ones. 

\begin{thm}
Let $n\geq2$ be an integer, let $F$ be a field with at least $(n-1)n+1$ elements where $char(F)$ does not divide $n$, and let $f\in F\langle X \rangle$ be a multilinear polynomial. If  $f(y_{1},\dots,y_{m-1},z)\notin \langle [y_{1},y_{2}] \rangle^{T_{\mathbb{Z}}}$ for every non neutral variable $z$, then $Im(f)$ on $M_{n}(F)$ contains $sl_{n}(F)$.
\end{thm}

\begin{proof}
Since $Im(f)$ is invariant under automorphisms, by Theorem \ref{tAR} it is enough to show that $Im(f)$ contains all matrices with zero diagonal. Let $A$ be a zero diagonal matrix and write $A$ as the sum of its homogeneous components
\[
A=\sum_{i=-n+1}^{-1}A_{i}+\sum_{i=1}^{n-1}A_{i}
\]
where $A_{i}=\displaystyle\sum_{k=1}^{n+i}a_{k-i,k}e_{k-i,k}$ for $i=-n+1$, \dots, $-1$ and $A_{i}=\displaystyle\sum_{k=1}^{n-i}a_{k,k+i}e_{k,k+i}$ for $i=1$, \dots, $n-1$.

By hypothesis $f(y_{1},\dots,y_{m-1},z^{(i)})$ is not a graded polynomial identity for $UT_{n}$ with the natural $\mathbb{Z}$-grading, for every variable $z^{(i)}$ of homogeneous degree $i$ where $1\leq i\leq n-1$. 

We now consider the following evaluations on generic matrices: $y_{j}=\displaystyle\sum_{k=1}^{n}w_{k}^{(j)}e_{kk}$ for all $j=1$, \dots, $m-1$ and $z^{(i)}=\displaystyle\sum_{k=1}^{n-i}w_{k}^{(m,i)}e_{k,k+i}$. 

Hence
\[
f(y_{1},\dots,y_{m-1},z^{(i)})=\sum_{k=1}^{n-i}p_{k,i}w_{k}^{(m,i)}e_{k,k+i}
\]
where $p_{k,i}$ is a polynomial in the variables $w_{k}^{(j)}$. Since $f\notin Id^{gr}(UT_{n})$, Corollary \ref{cnaturalUTn} gives us that the image of $f(y_{1},\dots,y_{m-1},z^{(i)})$ on $UT_{n}$ is exactly $(UT_{n})_{i}$. Hence all $p_{k,i}$ are nonzero polynomials. Moreover note that $p_{k,i}$ is such that all its monomials are multilinear ones. 

Analogously, we also have that $f(y_{1},\dots,y_{m-1},z^{(i)})$ is not a graded polynomial identity for the lower triangular matrix algebra $LT_{n}$ endowed with the natural $\mathbb{Z}$-grading, for $i=-n+1$, \dots, $-1$. Therefore 
\[
f(y_{1},\dots,y_{m-1},z^{(i)})=\sum_{k=1}^{n-i}q_{k,i}w_{k}^{(m,-i)}e_{k+i,k}
\]
where $z^{(i)}=\displaystyle\sum_{k=1}^{n+i}w_{k}^{(m,-i)}e_{k-i,k}$ and $q_{k,i}$ are nonzero commutative polynomials with multilinear monomials.

The number of polynomials $p_{k,i}$ and $q_{k,i}$ is exactly $(n-1)n$. We now apply Corollary \ref{ccomutpolyfinite} to get an evaluation of all variables $w_{k}^{(j)}$ such that the polynomials $p_{k,i}$ and $q_{k,i}$ assume simultaneously nonzero values in $F$. Such evaluations give us diagonal matrices $D_{1}$, \dots, $D_{m}$ such that  
\[
f(D_{1},\dots,D_{m-1},z^{(i)})=\sum_{k=1}^{n-i}\alpha_{i,k}w_{k}^{(m,i)}e_{k,k+i}
\]
where $\alpha_{i,k}$ are nonzero scalars. Therefore each matrix $A_{i}$ can be realized as $f(D_{1},\dots,D_{m-1},B_{i})$ for a suitable matrix $B_{i}\in (UT_{n})_{i}$, for every $i=1$, \dots, $n-1$. Similarly we also have that each matrix $A_{i}$ can be realized as $f(D_{1},\dots,D_{m-1},C_{i})$ for a suitable matrix $C_{i}\in (LT_{n})_{i}$, for all $i=-n+1$, \dots, $-1$. Hence
\[
A=\sum_{i=-n+1}^{-1}A_{i}+\sum_{i=1}^{n-1}A_{i}=\sum_{i=-n+1}^{-1}f(D_{1},\ldots,D_{m-1},C_{i})+\sum_{i=1}^{n-1}f(D_{1},\dots,D_{m-1},B_{i})
\]
and it is enough to use the linearity of $f$ in one variable to get $A\in Im(f)$.
\end{proof}

\section{The low dimension cases}
\label{sect6}

\subsection{Arbitrary gradings on $UT_{2}$ and $UT_{3}$}

We start this section with the following proposition.

\begin{prp}\label{lowprop}
Let $G$ be a group and let $\mathcal{A}$ and $\mathcal{B}$ be two $G$-graded algebras such that $\mathcal{B}$ is a graded homomorphic image of $\mathcal{A}$. Let $f\in F\langle X \rangle^{gr}$ be a graded  polynomial and assume that $f(\mathcal{A})$ is a homogeneous subspace of $\mathcal{A}$. Then $f(\mathcal{B})$ is also a homogeneous subspace of $\mathcal{B}$. 
\end{prp} 

\begin{proof}
Let $\phi\colon \mathcal{A}\rightarrow \mathcal{B}$ be a graded epimorphism. We start by noting that $\phi(f(\mathcal{A}))=f(\phi(\mathcal{A}))$. Then taking $b_{i}^{(1)}$, $b_{i}^{(2)}\in \mathcal{B}_{g_{i}}$, $i=1$, \dots, $m$, we have $b_{i}^{(j)}=\phi(a_{i}^{(j)})$ for some $a_{i}^{(j)}\in \mathcal{A}_{g_{i}}$, since $\phi$ is surjective. This leads us to
\begin{align*}
&\alpha f(b_{1}^{(1)},\dots,b_{m}^{(1)})+f(b_{1}^{(2)},\dots,b_{m}^{(2)})\\
&=\alpha f(\phi(a_{1}^{(1)}),\dots,\phi(a_{m}^{(1)}))+f(\phi(a_{1}^{(2)}),\dots,\phi(a_{m}^{(2)}))\\
&=\phi(\alpha f(a_{1}^{(1)},\dots,a_{m}^{(1)})+f(a_{1}^{(2)},\dots,a_{m}^{(2)}))
\end{align*}
which is an element from $\phi(f(\mathcal{A}))=f(\phi(\mathcal{A}))$ since we are assuming that $f(\mathcal{A})$ is a subspace. Then $f(\mathcal{B})$ is also a subspace.

Now we assume that $f(\mathcal{A})$ is a homogeneous subspace and let $b=b_{h_{1}}+\cdots + b_{h_{k}}$ be an element in $f(\mathcal{B})$ written as the sum of its homogeneous components. Since $b\in f(\mathcal{B})$ let 
\[
b=f(b_{1},\dots,b_{m})=f(\phi(a_{1}),\dots,\phi(a_{m}))=\phi(f(a_{1},\dots,a_{m}))
\]
for some $b_{i}=\phi(a_{i})$, and let 
\[
f(a_{1},\dots,a_{m})=a_{g_{1}}+\cdots + a_{g_{l}}
\]
be the sum of its homogeneous components. It follows that 
\[
b=\phi(a_{g_{1}})+\cdots +\phi(a_{g_{l}})
\]
and since $\phi$ is a graded homomorphism we must have $k=l$ and every $g_{t}$ must be equal to some $h_{s}$. Without loss of generality, we assume $b_{g_{t}}=\phi(a_{g_{t}})$. Now it is enough to use that $f(\mathcal{A})$ is homogeneous and $\phi(f(\mathcal{A}))=f(\phi(\mathcal{A}))$.
\end{proof}

\begin{rem}\label{remarkprop}
In the proof of Proposition \ref{lowprop} we have not used the associativity of $\mathcal{A}$. Therefore it also holds for arbitrary algebras, in particular an analogous proposition holds for graded Jordan algebras. 
\end{rem}

As a consequence of Proposition \ref{lowprop} we have the following theorem.

\begin{thm}\label{theoremUT2}
Let $UT_{2}=\mathcal{A}=\bigoplus_{g\in G}\mathcal{A}_{g}$ be some grading on $\mathcal{A}$ and let $f\in F\langle X \rangle^{gr}$ be a multilinear graded polynomial. Then $f(\mathcal{A})$ is a homogeneous subspace of $\mathcal{A}$.
\end{thm}

\begin{proof}
By Theorem \ref{gradingsupper} and Proposition \ref{lowprop}, it is enough to consider images of multilinear graded polynomials on elementary gradings only. We note that just two elementary $G$-gradings can be defined on $\mathcal{A}=UT_{2}$. Indeed, an elementary grading on $UT_{2}$ is completely determined by the homogeneous degree of $e_{12}$. If $\deg(e_{12})=1$, then we have the trivial grading, and we apply Theorem \ref{TLuoWang}. Hence we assume  $\mathcal{A}_{1}=span\{e_{11},e_{22}\}$ and $\mathcal{A}_{g}=span\{e_{12}\}$, where $g\neq1$. In this grading the images of multilinear polynomials in neutral variables are handled by Lemma \ref{sumofcoeffi} and Proposition \ref{basicprop}(2). Since $\mathcal{A}_{g}^{2}=\{0\}$, it is enough to consider multilinear polynomials in one variable of homogeneous degree $g$ and all remaining variables of neutral degree. In this case the image is contained in $\mathcal{A}_{g}$ and by Proposition \ref{basicprop}(1) we are done.
\end{proof}

Now we prove an analogous fact to Theorem \ref{theoremUT2} with $\mathcal{A}=UT_{3}$ instead of $UT_{2}$. From now on in this subsection we assume that $\mathcal{A}$ is endowed with some elementary $G$-grading given by a tuple $(g_{1},g_{2})\in G^{2}$. Hence $g_{1}=\deg(e_{12})$, $g_{2}=\deg(e_{23})$, and $g_{3}:=g_{1}g_{2}=\deg(e_{13})$.

Hence the elementary gradings on $UT_{3}$ are exactly the ones given by following relations.

\begin{itemize}
\item[(I)] $\{1\}\cap \{g_{1},g_{2},g_{3}\}\neq \emptyset$.
\begin{itemize}
\item[(a)] $g_{1}=g_{2}=1$, which implies $g_{3}=1$;
\item[(b)] $g_{1}=1$, which implies $g_{2}=g_{3}$;
\item[(c)] $g_{2}=1$, which implies $g_{1}=g_{3}$;
\item[(d)] $g_{3}=1$ and $g_{1}=g_{2}$;
\item[(e)] $g_{3}=1$ and $g_{1}\neq g_{2}$.
\end{itemize}
\item[(II)] $\{1\}\cap \{g_{1},g_{2},g_{3}\}=\emptyset$.
\begin{itemize}
\item[(a)] $1,g_{1},g_{2},g_{3}$ are pairwise distinct elements;
\item[(b)] $g_{1}=g_{2}\neq g_{3}$.
\end{itemize}
\end{itemize}

In the following lemmas we discuss the grading on $UT_{3}$ determined by each relation above and the respective image of a multilinear graded polynomial on such a graded algebra. 

\begin{lem}\label{l1UT3}
Let $UT_{3}$ be endowed with the grading (I)(b). Then $Im(f)$ on $UT_{3}$ is a homogeneous subspace.
\end{lem}

\begin{proof}
We denote $g_{2}=g$, then we have $\mathcal{A}_{1}=span\{e_{11},e_{22},e_{33},e_{12}\}$ and $\mathcal{A}_{g}=span\{e_{13},e_{23}\}$. Note that $\mathcal{A}_{g}^{2}=\{0\}$ and hence we only need to analyse multilinear polynomials in at most one variable of homogeneous degree $g$. 

The case when $f$ is a multilinear polynomial in neutral variables is settled by Lemma \ref{sumofcoeffi} and Proposition \ref{basicprop}(2).

Now we consider $f$ as a multilinear polynomial in one non neutral variable and $m-1$ neutral ones. Since $\mathcal{A}$ satisfies the graded identity $z[y_{1},y_{2}]\equiv 0$, then modulo this identity we write $f$ as
\[
\sum_{1\leq i_{1}<\cdots< i_{k}\leq m-1}h_{i_{1},\dots,i_{k}}z_{m}y_{i_{1}}\cdots y_{i_{k}}.
\]
If all polynomials $h_{i_{1}, \dots, i_{k}}$ have commutator degree different from $0$, then $Im(f)\subset span\{e_{13}\}$ and then we apply Proposition \ref{basicprop}(2). Otherwise we may assume, without loss of generality, that $h_{1,\dots,k}$ has commutator degree $0$. Then we perform the following evaluations: $y_{1}=\cdots=y_{k}=e_{33}$, $y_{j}=e_{11}+e_{22}$ for every $j\notin\{1,\dots,k\}$, and $z_{m}=\alpha^{-1}(a_{1}e_{13}+a_{2}e_{23})$, where $\alpha$ is the sum of the coefficients of $h_{1,\dots,k}$. Note that under such an evaluation we have $a_{1}e_{13}+a_{2}e_{23}\in Im(f)$ which proves that $Im(f)=\mathcal{A}_{g}$.
\end{proof}

\begin{lem}
Let $UT_{3}$ be endowed with the grading (I)(c). Then $Im(f)$ on $UT_{3}$ is a homogeneous subspace.
\end{lem}

\begin{proof}
Note that $\mathcal{A}_{1}=span\{e_{11},e_{22},e_{33},e_{23}\}$, $\mathcal{A}_{g_{1}}=span\{e_{12},e_{13}\}$ and also that $\mathcal{A}$ satisfies the identities $[y_{1},y_{2}]z\equiv0$ and $z_{1}z_{2}\equiv0$. Thus, the proof is similar to the one for (I)(b).
\end{proof}

\begin{lem}
Let $UT_{3}$ be endowed with the grading (I)(e). Then $Im(f)$ on $UT_{3}$ is a homogeneous subspace.
\end{lem}

\begin{proof}
Here we must have $\mathcal{A}_{1}=span\{e_{11},e_{22},e_{33},e_{13}\}$,  $\mathcal{A}_{g_{1}}=span\{e_{12}\}$, $\mathcal{A}_{g_{2}}=span\{e_{23}\}$. Note that $\mathcal{A}_{g_{1}}^{2}=\mathcal{A}_{g_{2}}^{2}=\{0\}$,  $\mathcal{A}_{g_{2}}\mathcal{A}_{g_{1}}=\{0\}$, and $\mathcal{A}_{g_{1}}\mathcal{A}_{g_{2}}\subset span\{e_{13}\}$.  

The case when $f$ is a multilinear polynomial in neutral variables can be treated as in the grading (I)(b). 
Hence we may consider $f$ is a multilinear polynomial in: one variable of degree $g_{1}$ (respectively $g_{2}$) and $m-1$ neutral variables, or in one variable of degree $g_{1}$, one of degree $g_{2}$ and $m-2$ neutral ones. In each of these situations we have that $Im(f)$ is contained in a 
one-dimensional space and we apply Proposition \ref{basicprop}(2). 
\end{proof}

\begin{lem}
Let $UT_{3}$ be endowed with the grading (II)(a). Then $Im(f)$ on $UT_{3}$ is a homogeneous subspace.
\end{lem}

\begin{proof}
We have $\mathcal{A}_{1}=span\{e_{11},e_{22},e_{33}\}$, $\mathcal{A}_{g_{1}}=span\{e_{12}\}$, $\mathcal{A}_{g_{2}}=span\{e_{23}\}$, and $\mathcal{A}_{g_{3}}=span\{e_{13}\}$. The only non trivial relation among the non neutral homogeneous components is given by $\mathcal{A}_{g_{1}}\mathcal{A}_{g_{3}}=\mathcal{A}_{g_{2}}$. 

The case of $f$ in neutral variables is the same as for the grading (I)(b).

Since the non neutral components are one-dimensional, then the image of a multilinear polynomial in one non neutral variable and $m-1$ neutral ones is always zero or the respective homogeneous component. 

In case $f$ has one variable of homogeneous degree $g_{1}$, one of degree $g_{3}$ and $m-2$ neutral ones then the image is contained in $\mathcal{A}_{g_{2}}$, and we are done.
\end{proof}

\begin{lem}
Let $UT_{3}$ be endowed with the grading (II)(b). Then $Im(f)$ on $UT_{3}$ is a homogeneous subspace.
\end{lem}

\begin{proof}
Note that $\mathcal{A}_{1}=span\{e_{11},e_{22},e_{33}\}$, $\mathcal{A}_{g_{1}}=span\{e_{12},e_{23}\}$ and $\mathcal{A}_{g_{3}}=span\{e_{13}\}$. We   only need to consider the case when $f$ is a multilinear polynomial in $m-1$ neutral variables and one of homogeneous degree $g_{1}$, since the remaining cases can be treated as above. We write $f=\displaystyle\sum_{j=1}^{m}f_{j}$ where $f_{j}$ is the sum of all monomials from $f$ which contain the variable $z_{m}$ in the $j$-th position. Hence, modulo $[y_{1},y_{2}]\equiv 0$ we have
\[
f_{j}=\sum_{1\leq i_{1}<\cdots<i_{j-1}\leq m-1}\alpha_{i_{1},\dots,i_{j-1}}y_{i_{1}}\cdots y_{i_{j-1}}z_{m}y_{k_{1}}\cdots y_{k_{m-j}}
\]
where $k_{1}<\cdots<k_{m-j}$. We evaluate $y_{i}=w_{1}^{(i)}e_{11}+e_{22}+w_{3}^{(i)}e_{33}$ and $z_{m}=w_{1}^{(m)}e_{12}+w_{2}^{(m)}e_{23}$. Thus $f(y_{1},\dots,y_{m-1},z_{m})$ is given by
\[
\begin{pmatrix}
     0&p_{1}(w_{1}^{(1)},\dots,w_{1}^{(m-1)})w_{1}^{(m)} & 0\\
     &0 &p_{2}(w_{3}^{(1)},\dots,w_{3}^{(m-1)})w_{2}^{(m)} \\
     & &0
\end{pmatrix}
\]
where $p_{1}(w_{1}^{(1)},\dots,w_{1}^{(m-1)})=\displaystyle\sum_{j=1}^{m}\sum_{1\leq i_{1}<\cdots<i_{j-1}\leq m-1}\alpha_{i_{1},\dots,i_{j-1}}w_{1}^{(i_{1})}\cdots w_{1}^{(i_{j-1})}$ and $p_{2}$ is given analogously.  

We claim that $p_{1}$ takes nonzero values on $F$. Indeed, assume that $p_{1}$ is a polynomial identity for $F$ and denote $e_{j}=\displaystyle\sum_{1\leq i_{1}<\cdots< i_{j-1}\leq m-1}\alpha_{i_{1},\dots,i_{j-1}}w_{1}^{(i_{1})}\cdots w_{1}^{(i_{j-1})}$. Note that $e_{1}\in F$ and taking $w_{1}^{(1)}=\cdots=w_{1}^{(m-1)}=0$ we have $e_{1}=0$. Taking $w_{1}^{(l)}=1$ and zero for the remaining values of $w_{1}$ we have $\alpha_{l}=0$ for all $l\in\{1,\dots,m-1\}$ and hence $e_{2}=0$. Now assume $e_{l}=0$ for all $l<k$, we shall prove that $e_{k}=0$. For each chosen $i_{1}$, \dots, $i_{k-1}$ we take $w_{1}^{(r)}=0$ for all $r\notin\{i_{1},\dots,i_{k-1}\}$, then  $e_{l}=0$ for all $l>k$ and $e_{k}=\alpha_{i_{1},\dots,i_{k-1}}w_{1}^{(i_{1})}\cdots w_{1}^{(i_{k-1})}$. Then we take $w_{1}^{(i_{1})}=\cdots=w_{1}^{(i_{k-1})}=1$ and we conclude that $\alpha_{i_{1},\dots,i_{k-1}}=0$. Hence $p_{1}=0$, which is a contradiction. Analogous claim holds for $p_{2}$. Therefore it is enough to use the variables $w_{1}^{(m)}$ and $w_{2}^{(m)}$ to realize any matrix in $\mathcal{A}_{g_{1}}$ in the image of $f$. 
\end{proof}

\begin{lem}
Let $UT_{3}$ be endowed with the grading (I)(d). Then $Im(f)$ on $UT_{3}$ is a homogeneous subspace.
\end{lem}

\begin{proof}
We denote $g_{1}=g$ and note that  $\mathcal{A}_{1}=span\{e_{11},e_{22},e_{33},e_{13}\}$ and $\mathcal{A}_{g}=span\{e_{12},e_{23}\}$. Then $\mathcal{A}_{g}^{2}\subset span\{e_{13}\}$ and $\mathcal{A}$ satisfies the identities $z[y_{1},y_{2}]\equiv 0$ and $[y_{1},y_{2}]z\equiv 0$. The case when $f$ has one variable of homogeneous degree $g$ and $m-1$ neutral variables can be treated as in the previous lemma. The remaining cases are considered as above.  
\end{proof}

Hence we have the following theorem.

\begin{thm}
Let $F$ be an arbitrary field, let $UT_{3}=\mathcal{A}=\bigoplus_{g\in G}A_{g}$ be some non trivial grading on $\mathcal{A}$, and let $f\in F\langle X \rangle^{gr}$ be a multilinear graded  polynomial. Then $Im(f)$ on $\mathcal{A}$ is a homogeneous subspace of $\mathcal{A}$. If $|F|\geq 3$ and $\mathcal{A}$ is equipped with the trivial grading, then the image is also a subspace. 
\end{thm}

\begin{proof}
The proof is clear from the previous lemmas and Proposition \ref{lowprop}.
\end{proof}

\subsection{The graded Jordan algebra $UJ_{2}$}

Throughout this subsection we assume that $F$ is a field of characteristic different from $2$ and we denote by $UJ_{n}$ the Jordan algebra of the upper triangular matrices with product $a\circ b=ab+ba$. Unlike the associative setting, gradings on $UJ_{n}$ are not only elementary ones. Actually, a second kind of gradings also occurs on $UJ_{n}$, the so-called mirror type gradings, and we define these below. First of all let us introduce the following notation. 

Let $i$, $m$ be non negative integers and set 
\[
E_{i:m}^{+}=e_{i,i+m}+e_{n-i-m+1,n-i+1} \ \mbox{and} \ E_{i:m}^{-}=e_{i,i+m}-e_{n-i-m+1,n-i+1}.
\]
\begin{defi}
A $G$-grading on $UJ_{n}$ is called of mirror type if the matrices $E_{i:m}^{+}$ and $E_{i:m}^{-}$ are homogeneous, and $\deg(E_{i:m}^{+})\neq \deg(E_{i:m}^{-})$.
\end{defi}

We recall the following theorem from \cite{KYa}.

\begin{thm}[\cite{KYa}]
The $G$-gradings on the Jordan algebra $UJ_{n}$ are, up to a graded isomorphism, elementary or of mirror type.
\end{thm}

In particular we have the following classification of the gradings on $UJ_{2}$.

\begin{prp}
Up to a graded isomorphism, the gradings on $UJ_{2}$ are given by $UJ_{2}=\mathcal{A}=\bigoplus_{g\in G}\mathcal{A}_{g}$ where

\begin{itemize}
\item[(I)] elementary ones
\begin{itemize}
\item[(a)] trivial grading;
\item[(b)]  $\mathcal{A}_{1}=\begin{pmatrix}
    a & 0 \\
     & b
\end{pmatrix}$, $\mathcal{A}_{g}=\begin{pmatrix}
    0 & c \\
     & 0
\end{pmatrix}$
\end{itemize}
\item[(II)] mirror type ones;
\begin{itemize}
\item[(a)] $\mathcal{A}_{1}=\begin{pmatrix}
    a & 0 \\
     & a
\end{pmatrix}$, $\mathcal{A}_{g}=\begin{pmatrix}
    b & c \\
     & -b
\end{pmatrix}$
\item[(b)] $\mathcal{A}_{1}=\begin{pmatrix}
    a & b \\
     & a
\end{pmatrix}$, $\mathcal{A}_{g}=\begin{pmatrix}
    c & 0 \\
     & -c
\end{pmatrix}$
\item[(c)] $\mathcal{A}_{1}=\begin{pmatrix}
    a & 0 \\
     & a
\end{pmatrix}$, $\mathcal{A}_{g}=\begin{pmatrix}
    b & 0 \\
     & -b
\end{pmatrix}$, $\mathcal{A}_{h}=\begin{pmatrix}
    0 & c \\
     & 0
\end{pmatrix}$
\end{itemize}
\end{itemize}
where $g$, $h\in G$ are elements of order $2$.
\end{prp}

In \cite{KYa} it was also proved that the support of a grading on $UJ_{n}$ is always abelian (see \cite{KYa} Theorem 24). Hence by Proposition \ref{basicprop} (4) we have that $Im(f)$ on $UJ_{n}$ is a homogeneous subset for any multilinear graded polynomial $f\in J(X)$.

Next we analyse the images of a multilinear graded Jordan polynomial $f$ on the gradings considered above.

\begin{lem}\label{l1jordan}
Let $UJ_{2}$ be endowed with the grading (I)(b). Then $Im(f)$ on $UJ_{2}$ is a homogeneous subspace.
\end{lem}

\begin{proof}
We start with a multilinear polynomial $f$ in $m$ neutral variables. We evaluate each variable $y_{i}$ to an arbitrary diagonal matrix $D_{i}$. Therefore each monomial $\mathbf{m}$ in $f$ is evaluated to $2^{m-1}\beta D_{1}\cdots D_{m}$, where $\beta\in F$ is the coefficient of $\mathbf{m}$. Hence 
\[
f(D_{1},\dots,D_{m})=2^{m-1}\alpha D_{1}\cdots D_{m}
\]
where $\alpha\in F$ is the sum of all coefficients of $f$. In case $\alpha=0$, then $f=0$ is a graded polynomial identity for $UJ_{2}$, otherwise we can take $D_{2}=\cdots=D_{m}=I_{2}$ and use $D_{1}$ in order to obtain every diagonal matrix in the image of $f$.

Since $UJ_{2}$ satisfies the graded identity $z_{1}\circ z_{2}=0$ such that $\deg(z_{1})=\deg(z_{2})=g$, then we only need to analyse the case where $f$ is a multilinear polynomial in $m-1$ neutral variables and one of homogeneous degree $g$. Obviously we must have $Im(f)\subset \mathcal{A}_{g}$ and this homogeneous component is one-dimensional, then we are done.
\end{proof}

For the grading (II)(a) we recall a lemma from \cite{GSa} applied to multilinear polynomials. In order to make the notation more compact we omit the symbol $\circ$ for the Jordan product, and we write $ab$ instead of $a\circ b$. If no brackets are given in a product, we assume these left-normed, that is $abc=(ab)c$.

\begin{lem}[\cite{GSa}]\label{l1DimasSalomao}
Let $UJ_{2}$ be endowed with the grading (II)(a) and let $f\in J(X)_{g}$ be a multilinear $\mathbb{Z}_{2}$-graded polynomial. Then, modulo the graded identities of $UJ_{2}$, we can write $f$ as a linear combination of monomials of the type
\[
y_{1}\cdots y_{l}z_{i_{0}}(z_{i_{1}}z_{i_{2}})\cdots (z_{i_{2m-1}}z_{i_{2m}}), 1<\cdots<l, i_{1}<i_{2}<i_{3}<\cdots<i_{m}<i_{m+1}, i_{0}>0.
\]
\end{lem}

\begin{lem}\label{l2jordan}
Let $UJ_{2}$ be endowed with the grading (II)(a). Then $Im(f)$ on $UJ_{2}$ is a homogeneous subspace.
\end{lem}

\begin{proof}
Since $\dim\mathcal{A}_{1}=1$ it follows that if the image of a multilinear polynomial on $UJ_{2}$ is contained in $\mathcal{A}_{1}$ then it must be either $\{0\}$ or $\mathcal{A}_{1}$.

Now we consider a multilinear polynomial $f$ in homogeneous variables of degree $1$ and $g$ such that $\deg f=g$. Let $\mathbf{m}=y_{1}\cdots y_{l}z_{i_{0}}(z_{i_{1}}z_{i_{2}})\cdots (z_{i_{2m-1}}z_{i_{2m}})$ be a monomial as in Lemma \ref{l1DimasSalomao}. We note that the main diagonal of a matrix in $m(UJ_{2})$ is such that the entry $(k,k)$ is given by $(-1)^{k+1}2^{m+l+1}a$ where $a$ is the product of the entries at position $(1,1)$ of all matrices $y$ and $z$. Hence every matrix in $Im(f)$ is of the form
\[
\begin{pmatrix}
    2^{m+l+1}\alpha\cdot a & * \\
     & -2^{m+l+1}\alpha\cdot a
\end{pmatrix}
\]
where $\alpha$ is the sum of all coefficients of $f$.

In case $\alpha=0$, then $Im(f)\subset span\{e_{12}\}$ and then the image is completely determined. 

We consider now $\alpha\neq0$. Without loss of generality, we assume that the nonzero scalar occurs in the monomial $y_{1}\cdots y_{l}z_{0}(z_{1}z_{2})\cdots (z_{2m-1}z_{2m})$. Then we make the following evaluation: $y_{1}=\cdots=y_{l}=I_{2}$, $z_{0}=w_{1}(e_{11}-e_{22})+w_{2}e_{12}$ and $z_{i}=e_{11}-e_{22}$ for every $i=1$, \dots, $2m$, where $w_{1}$, $w_{2}$ are commutative variables. Therefore
\[
f(y_{1},\dots,y_{l},z_{0},\dots,z_{2m})=\begin{pmatrix}
    2^{m+l+1}\alpha w_{1} & 2^{m+l+1}w_{2} \\
     & -2^{m+l+1}\alpha w_{1}
\end{pmatrix}.
\]
Since $char(F)\neq 2$ and $\alpha\neq0$, it follows that $Im(f)=\mathcal{A}_{g}$. 
\end{proof}

Now we consider the grading (II)(b) and we recall another lemma from \cite{GSa}.

\begin{lem}[\cite{GSa}]\label{l2GSa}
Let $f\in J(X)_{1}$ be a multilinear polynomial. Then, modulo the graded identities of $UJ_{2}$, $f$ can be written as a linear combination of monomials of the form
\begin{enumerate}
    \item $(y_{i_{1}}\cdots y_{i_{r}})(z_{j_{1}}\cdots z_{j_{l}})$;
    \item $(((y_{i}z_{j_{1}})z_{j_{2}})y_{i_{1}}\cdots y_{i_{r}})z_{j_{3}}\cdots z_{j_{l}}$,
\end{enumerate}
where $l\geq 0$ is even, $r\geq 0$, $i_{1}<\cdots <i_{r}$, and $z_{j_{1}}<z_{j_{2}}<z_{j_{3}}<\cdots <z_{j_{l}}$.
\end{lem}

\begin{lem}\label{l3jordan}
Let $UJ_{2}$ be endowed with the grading (II)(b). Then $Im(f)$ on $UJ_{2}$ is a homogeneous subspace.
\end{lem}

\begin{proof}
We start with a multilinear polynomial $f$ in $m$ neutral variables. Note that a multilinear monomial of degree $m$ evaluated on $\mathcal{A}_{1}$ is a matrix whose main diagonal is given by $2^{m-1}aI_{2}$ where $a$ is the product of the entries on the main diagonal of the matrices used in the evaluation. Hence a matrix in $Im(f)$ must be of the form
\[
\begin{pmatrix}
    2^{m-1}\alpha\cdot a & * \\
     & 2^{m-1}\alpha\cdot a
\end{pmatrix}
\]
where $\alpha$ is the sum of all coefficients of $f$. In case $\alpha=0$, we have $Im(f)\subset span\{e_{12}\}$ and the image is completely determined. From now on we assume $\alpha\neq0$, and we evaluate $m-1$ matrices by $I_{2}$ and one matrix, say $y_{1}$, by $w_{1}(e_{11}+e_{22})+w_{2}e_{12}$ where $w_{1}$, $w_{2}$ are commuting variables. Therefore
\[
f(y_{1},\dots,y_{m})=\begin{pmatrix}
    2^{m-1}\alpha w_{1} & 2^{m-1}\alpha w_{2} \\
     & 2^{m-1}\alpha w_{1}
\end{pmatrix},
\]
and since $char(F)\neq2$ and $\alpha\neq0$, we have that $Im(f)=\mathcal{A}_{1}$.

Now we consider a multilinear polynomial $f$ which has at least one variable of homogeneous degree $g$. In case $\deg f=g$, the image $Im(f)$ is completely determined, since $\dim\mathcal{A}_{g}=1$. So we assume $\deg f=1$. In case $f$ is a multilinear polynomial in variables of homogeneous degree $g$, then $Im(f)$ is contained in the vector space of the scalars matrices, and therefore the image is completely determined. Hence we assume further that $f$ has at least one variable of neutral degree and let $f$ be  a multilinear polynomial in $l$ neutral variables $y_{1}$, \dots, $y_{l}$, and $m-l$ variables $z_{l+1}$, \dots, $z_{m}$ of homogeneous degree $g$. Then by Lemma \ref{l2GSa}, we write $f$ as 
\[
f=\alpha_{1}(y_{1}\cdots y_{l})(z_{l+1}\cdots z_{m})+\sum_{i=1}^{l}\alpha_{i+1}(((y_{i}z_{l+1})z_{l+2})y_{1}\cdots \widehat{y}_{i}\cdots  y_{l})z_{l+3}\cdots z_{m}.
\]
Here $\widehat{y}_{i}$ means that the variable $y_{i}$ does not appear in the product $y_{1}\cdots \widehat{y}_{i}\cdots  y_{l}$.

We replace $y_{i}=w_{1}^{(i)}(e_{11}+e_{22})+w_{2}^{(i)}e_{12}$ and $z_{j}=w_{1}^{(j)}(e_{11}-e_{22})$, where the $w$'s are commuting variables. Note that the Jordan product of two matrices $y_{1}$ and $y_{2}$ is given by $2y_{1}\cdot y_{2}$ where the dot $\cdot$ stands for the usual product of matrices. On the other hand, the usual product of $n$ matrices $y_{1}$,  \dots,  $y_{n}$ is given by 
\[
\begin{pmatrix}
    w_{1}^{(1)}\cdots w_{n}^{(n)} & w  \\
     & w_{1}^{(1)}\cdots w_{n}^{(n)} 
\end{pmatrix}.
\]
Here $w=\displaystyle\sum_{\substack{1\leq i_{1}<\cdots<i_{n-1}\leq n \\
i_{n}\in\{1,\dots,n\}\setminus\{i_{1},\dots,i_{n-1}\}}}w_{1}^{(i_{1})}\cdots w_{1}^{(i_{n-1})}w_{2}^{(i_{n})}$, as one can see by induction on $n$. Hence the image of the monomial $\alpha_{i+1}(((y_{i}z_{l+1})z_{l+2})y_{1}\cdots \widehat{y}_{i}\cdots  y_{l})z_{l+3}\cdots z_{m}$ is equal to
\[
2^{m-1}\alpha_{i+1}\begin{pmatrix}
   w_{1}^{(l)}\cdots w_{1}^{(l)}w_{1}^{(l+1)}\cdots w_{1}^{(m)}  & w_{i}w_{1}^{(l+1)}\cdots w_{1}^{(m)} \\
     & w_{1}^{(l)}\cdots w_{1}^{(l)}w_{1}^{(l+1)}\cdots w_{1}^{(m)}
\end{pmatrix}
\]
where $w_{i}$ is given as $w$ above but $i_{n}\neq i$. 
 
Therefore the main diagonal of $f(y_{1},\dots,y_{l},z_{l+1},\dots,z_{m})$ is given by
\[
2^{m-1}\alpha w_{1}^{(1)}\cdots w_{1}^{(m)}
\]
where $\alpha$ is the sum of all coefficients in $f$. The entry at position $(1,2)$ is 
 \[
 \displaystyle\sum_{\substack{1\leq i_{1}<\cdots<i_{l-1}\leq l \\
k\in\{1,\dots,l\}\setminus\{i_{1},\dots,i_{l-1}\}}}\beta_{k}w_{1}^{(i_{1})}\cdots w_{1}^{(i_{l-1})}w_{2}^{(k)}w_{1}^{(l+1)}\cdots w_{1}^{(m)}
 \]
 where $\beta_{k}=\alpha_{1}+\displaystyle\sum_{\substack{j=1 \\ j\neq k}}^{l}\alpha_{j+1}$.

If all $\beta_{k}$ are equal to zero, then $Im(f)$ is contained in the space of the scalar matrices and we are done. So we may assume that some of the $\beta_{k}$ is nonzero, and without loss of generality we suppose $\beta_{l}\neq0$. In this case we claim that the image will be the whole neutral component. Indeed, take $A=a_{1}(e_{11}+e_{22})+a_{2}e_{12}\in \mathcal{A}_{1}$. We evaluate the matrices of degree $g$ by $e_{11}-e_{22}$, that is, we take $w_{1}^{(l+1)}=\cdots=w_{1}^{(m)}=1$. We also evaluate the neutral matrices $y_{1}$, \dots, $y_{l-1}$ by the identity matrix, that is, we take $w_{1}^{(1)}=\cdots = w_{1}^{(l-1)}=1$ and $w_{2}^{(1)}=\cdots =w_{2}^{(l-1)}=0$. Hence the equality
\[
f(I_{2},\dots,I_{2},y_{l},e_{11}-e_{22},\dots,e_{11}-e_{22})=A
\]
 leads us to the following linear system
 \[
 \left\{\begin{array}{c}
     2^{m-1}\alpha w_{1}^{(l)}=a_{1}  \\
     2^{m-1}\beta_{l}w_{2}^{(l)}=a_{2} 
 \end{array}\right.
 \]
 which has $w_{1}^{(1)}=(2^{m-1}\alpha)^{-1}a_{1}$ and $w_{2}^{(2)}=(2^{m-1}\beta_{1})^{-1}a_{2}$ as the solution.
\end{proof} 
 
\begin{thm}
Let $UJ_{2}=\bigoplus_{g\in G}\mathcal{A}_{g}$ be a non trivial $G$-grading and let $f\in J(X)$ be a multilinear graded  Jordan polynomial. Then $Im(f)$ on $UJ_{2}$ is a homogeneous subspace. The same conclusion also holds for the trivial grading on $UJ_{2}$ for arbitrary fields.
\end{thm}
 
\begin{proof}
We first consider a non trivial grading on $UJ_{2}$. By Remark \ref{remarkprop} we may reduce the defined grading on $UJ_{2}$ to one of those described above. We note that the case of the grading (II)(c) follows from the fact that $Im(f)$ on $UJ_{2}$ is a homogeneous subset and all homogeneous components in this grading are one dimensional. We use Lemmas \ref{l1jordan},\ref{l2jordan} and \ref{l3jordan} for the remaining non trivial gradings. Now we consider the trivial grading on $UJ_{2}$. Let $f\in J(X)$ be a multilinear polynomial. We may assume that $f\notin Id(UJ_{2})$. By \cite{Sli}, the algebra $J(X)/Id(UJ_{2})$ is a special Jordan algebra,  and hence we may assume $f$ is an element in the free special Jordan algebra. Therefore, the image $Im(f)$ on $UJ_{2}$ is equal to the image of some associative polynomial on $UT_{2}$. Hence $Im(f)\in\{J,UJ_{2}\}$.
\end{proof}

\begin{rem}
Consider the Lie algebra $UT_{n}^{(-)}$ with product given by the Lie bracket. Given a grading on $UT_{n}^{(-)}$, note that $J=[UT_{n}^{(-)},UT_{n}^{(-)}]$ is always a homogeneous ideal. We also note that if $f\in L(X)^{gr}$ is a multilinear polynomial of degree $\ge 2$, then $Im(f)$ on $UT_{n}^{(-)}$ is contained in $J$. In particular, for $n=2$ we must have that $Im(f)$ is contained in $span\{e_{12}\}$ which is a homogeneous subspace. Since the image of multilinear polynomials of degree $1$ is trivial, we have that $Im(f)$ on the graded algebra $UT_{2}^{(-)}$ is always a homogeneous subspace, regardless of the grading defined on $UT_{2}^{(-)}$.
\end{rem}

\subsection{The natural elementary $\Z_{3}$-grading in the Jordan algebra $UJ_{3}$}

In this section we study images of multilinear polynomials on the Jordan algebra $\mathcal{A}=UJ_{3}$ endowed with the elementary $\Z_{3}$-grading given by the sequence $(\overline{0},\overline{1},\overline{2})$, that is, $\mathcal{A}_{\overline{0}}=span\{e_{11},e_{22},e_{33}\}$, $\mathcal{A}_{\overline{1}}=span\{e_{12},e_{23}\}$, and $\mathcal{A}_{\overline{2}}=span\{e_{13}\}$. 

We denote by $(x_{1},x_{2},x_{3})=(x_{1}x_{2})x_{3}-x_{1}(x_{2}x_{3})$ the associator of the elements $x_{1}$, $x_{2}$, $x_{3}$.

We recall the following identity which holds in any Jordan algebra.

\begin{lem}\label{lJacobson}
Let $\mathcal{J}$ be a Jordan algebra. Then
\[
abcd+adcb+bdca=(ab)(cd)+(ac)(bd)+(ad)(bc)
\]
for all $a$, $b$, $c$, $d\in \mathcal{J}$.
\end{lem}

\begin{proof}
See for example \cite[Page 34]{Jac}.
\end{proof}

As an easy consequence of Lemma \ref{lJacobson} we have 
\begin{equation}\label{iJacobson}
abcd+adcb+bdca=abdc+acdb+bcda
\end{equation}
for every $a$, $b$, $c$, $d\in \mathcal{J}$.

The next lemma point out some graded identities for the algebra $UJ_{3}$.

\begin{lem}\label{lidentitiesjordan}
The identities 
\[
(y_{1},y_{2},y_{3})\equiv 0, (y_{1},z,y_{2})\equiv 0 \ \mbox{and} \ z_{1}z_{2}\equiv 0
\]
hold for $UJ_{3}$, where $z$ is an odd variable and $\deg(z_{1})+\deg(z_{2})=\overline{0}$.
\end{lem}

\begin{proof}
A straightforward computation, hence omitted.
\end{proof}

The next lemma has the same proof as \cite[Lemma 5.3]{GSa}. However we will consider its proof here for the sake of completeness.

\begin{lem}
The polynomial 
\[
g=y_{1}(y_{2}(y_{3}z))-\frac{1}{2}\bigg(y_{1}(z(y_{2}y_{3})) +y_{2}(z(y_{1}y_{3}))+y_{3}(z(y_{1}y_{2}))-z(y_{1}(y_{2}y_{3}))\bigg)
\]
is a consequence of $(y_{1},z,y_{2})$, where $\deg(z)\in\{\overline{1},\overline{2}\}$.
\end{lem}

\begin{proof}
By identity (\ref{iJacobson}) we have
\[
-((y_{2}y_{3})z)y_{1}-((y_{1}y_{3})z)y_{2}-((y_{1}y_{2})z)y_{3}+((y_{2}y_{3})y_{1})z=-((y_{2}z)y_{1})y_{3}-((y_{3}z)y_{1})y_{2}.
\]
Hence, we can write $h=2g$ as
\begin{align*}
 h= & \ 2(y_{1}(y_{2}(y_{3}z))-((y_{2}z)y_{1})y_{3}-((y_{3}z)y_{1})y_{2}\\
 = & \ (y_{3},z,y_{2})y_{1}+(y_{2},zy_{3},y_{1})+(y_{3},zy_{2},y_{1})
\end{align*}
which implies that $g$ is a consequence of $(y_{1},z,y_{2})$.
\end{proof}

Given two even variables $y_{i}$ and $y_{j}$ we set $y_{i}<y_{j}$ if $i<j$. Hence we define an order on words in even variables $Y_{1}<Y_{2}$ considering the left lexicographic order in case $Y_{1}$ and $Y_{2}$ have the same length, and $Y_{1}<Y_{2}$ in case $Y_{2}$ is longer than $Y_1$.  For the next lemma we use ideas from \cite[Lemma 5.6]{GSa}. We denote by $T$ the $T$-ideal generated by the identities from Lemma \ref{lidentitiesjordan}.

\begin{lem}\label{ljordanUJ3}
Let $f=f(y_{1},\dots,y_{m-1},z_{m})\in J(X)$ be a multilinear polynomial, where $\deg(z_{m})\in \{\overline{1},\overline{2}\}$. Then modulo $T$, $f$ is a linear combination of monomials of the form $Y_{1}(zY_{2})$, where each $Y_{i}$ is an increasingly ordered product of even variables and $Y_{1}<Y_{2}$. 
\end{lem}

\begin{proof}
It is enough to consider $f=f(y_{1},\dots,y_{m-1},z_{m})$ as a monomial. We apply induction on $m$. If $m=1$ or $m=2$, then the conclusion is obvious. So we assume $m\geq 3$ and we write $f=gh$ where $g$, $h\in J(X)$. Without loss of generality we may assume that the odd variable $z_{m}$ occurs in $g$. Hence $h=Y_{1}$ and by the induction hypothesis we must have $g$ as a linear combination of monomials of the form $Y_{2}(zY_{3})$. On the other hand, we have
\[
(Y_{2}(zY_{3}))Y_{1}=\frac{1}{2}\bigg(Y_{1}(z(Y_{2}Y_{3}))+Y_{2}(z(Y_{1}Y_{3})+Y_{3}(z(Y_{1}Y_{2}))-z(Y_{1}(Y_{2}Y_{3}))\bigg).
\]
Now it is enough to use the identities $(y_{1},y_{2},y_{3})\equiv0$,  $(y_{1},z,y_{2})\equiv0$ and the commutativity of the Jordan product to get the desired conclusion. 
\end{proof}

Now we are ready to prove the main theorem of this subsection.

\begin{thm}
Let $F$ be an infinite field of characteristic different from 2 and let $f\in J(X)$ be a multilinear graded  polynomial. Then the image of $f$ on the graded Jordan algebra $UJ_{3}$ endowed with the natural elementary $\Z_{3}$-grading is either $\{0\}$ or some homogeneous component.
\end{thm}

\begin{proof}
Since $f$ is a homogeneous polynomial in the graded algebra $J(X)$ and $z_{1}z_{2}\equiv0$ holds on $UJ_{3}$, for $\deg z_{1}+\deg z_{2}=\overline{0}$, we will consider the following three cases in our proof.

{\it Case 1:} $\deg f=\overline{0}$. Here we must have $f=f(y_{1},\dots,y_{m})$ and the proof is the same as the first paragraph of the proof of Lemma \ref{l1jordan}. 

{\it Case 2:} $\deg f=\overline{1}$. Let $f=f(y_{1},\dots,y_{m-1},z_{m})$ be such that $\deg z_{m}=\overline{1}$. By Lemma \ref{ljordanUJ3}, modulo $T$, we may write $f$ as a linear combination of monomials of the form $Y_{1}(z_{m}Y_{2})$, where $Y_{1}<Y_{2}$. On the other hand, given 
\begin{equation}\label{J3evaluation}
y_{i}=\sum_{k=1}^{3}w_{k}^{(i)}e_{kk} \ \mbox{and} \ z_{m}=w_{1}^{(m)}e_{12}+w_{2}^{(m)}e_{23},
\end{equation}
note that $y_{i}z_{m}=\begin{pmatrix}
     0&(w_{1}^{(i)}+w_{2}^{(i)})w_{1}^{(m)} &0  \\
     &0 &(w_{2}^{(i)}+w_{3}^{(i)})w_{2}^{(m)}  \\ 
     & &0
\end{pmatrix}$ and then
\[
f(y_{1},\dots,y_{m-1},z_{m})=\begin{pmatrix}
     0&p_{1}w_{1}^{(m)} &0  \\
     &0 &p_{2}w_{2}^{(m)}  \\ 
     & &0
\end{pmatrix}
\]
where $p_{1}$ and $p_{2}$ are polynomials in the variables $w^{(i)}$, $i=1$, \dots, $m-1$. We claim that if $f\neq0$ modulo $T$, then $p_{1}\neq0$ and $p_{2}\neq0$. Indeed, consider the monomial $\mathbf{m}=\alpha Y_{1}(z_{m}Y_{2})$, where   $Y_{1}=y_{j_{1}}\cdots y_{j_{r}}$, $Y_{2}= y_{l_{1}}\cdots y_{l_{s}}$ and $Y_{1}<Y_{2}$. Note that the $(1,2)$ entry of the image of $\mathbf{m}$ under the evaluation (\ref{J3evaluation}) is given by
\[
2^{a}\alpha (w_{1}^{(j_{1})}\cdots w_{1}^{(j_{r})}+w_{2}^{(j_{1})}\cdots w_{2}^{(j_{r})})w_{1}^{(m)}(w_{1}^{(l_{1})}\cdots w_{1}^{(l_{s})}+w_{2}^{(l_{1})}\cdots w_{2}^{(l_{s})}) 
\]
for some power of $2$. Hence $p_{1}$ contains the following monomials 
\[
2^{a}\alpha w_{1}^{(j_{1})}\cdots w_{1}^{(j_{r})}w_{2}^{(l_{1})}\cdots w_{2}^{(l_{s})} \ \mbox{and} \ 2^{a}\alpha w_{2}^{(j_{1})}\cdots w_{2}^{(j_{r})}w_{1}^{(l_{1})}\cdots w_{1}^{(l_{s})}.
\]
Since $Y_{1}<Y_{2}$, the two monomials above can only be obtained from the monomial $\mathbf{m}$. Hence, if $f\neq 0$ modulo $T$, then $f$ contains some monomial $\mathbf{m}$ as above for some nonzero $\alpha$, which will imply in nonzero monomials in $p_{1}$ that are not multiple of any other one that comes from the remaining monomials of $f$. The same ideas also prove that $p_{2}\neq0$.

Now we use the fact that $F$ is infinite to get evaluations of the even variables for diagonal matrices such that $p_{1}$ and $p_{2}$ assume nonzero values on $F$, simultaneously. We finally use the variables $w_{1}^{(m)}$ and $w_{2}^{(m)}$ to get arbitrary odd matrices in $Im(f)$, that is, $Im(f)=(UJ_{3})_{\overline{1}}$.

{\it Case 3:} $\deg f=\overline{2}$. This last case follows from the fact that the homogeneous component of degree $\overline{2}$ is one dimensional. 
\end{proof}

\section*{Funding}

P. Fagundes was supported by S\~ao Paulo Research Foundation (FAPESP), grant \# 2019/16994-1. P. Koshlukov was partially supported by S\~ao Paulo Research Foundation (FAPESP), grant \# 2018/23690-6 and by CNPq grant
\# 302238/2019-0.

\end{document}